\documentclass{amsart}
\usepackage{amssymb}
\usepackage{graphicx}
\usepackage[mathscr]{euscript}
\def\({\left(}
\def\){\right)}

\newtheorem{lema}{Lemma}[section]

\newtheorem*{teorema*}{Theorem}

\newtheorem{remark}[lema]{Remark}

\newtheorem{example}[lema]{Example}

\newtheorem{corollary}[lema]{Corollary}
\newtheorem{theorem}[lema]{Theorem}

\newtheorem{proposition}[lema]{Proposition}
\newtheorem{note}[lema]{Note}
\newtheorem{definition}[lema]{Definition}

  {\par \hfill \fbox{}}
  {\par \hfill \fbox{}}
  {\par \hfill }
  {\par \hfill }

\def\beq{\begin{equation}}
\def\eeq{\end{equation}}

\def\epsilon{\varepsilon}

\begin{document}

\title{On nano $h$-open sets}

\author{Shallu Sharma}
\address{1-Department of Mathematics, University of Jammu} 
\email{shallujamwal09@gmail.com}

\author{Pooja Saproo}
\address{2-Department of Mathematics, University of Jammu} 
\email{poojasaproo1994@gmail.com}

\author{Naresh Digra}
\address{2-Department of Mathematics, University of Jammu} 
\email{nareshdigra79@gmail.com}

\author{Iqbal Kour}
\address{2-Department of Mathematics, University of Jammu} 
\email{iqbalkour208@gmail.com}

\thanks{}
\date{}

\begin{abstract}
The main aspect of this paper is to introduce a new generalisation of nano open sets namely, nano $h$-open sets. These newly generalised sets serve as the foundation for the definition of nano $h$-continuous functions and some results involving their characterizations are established. Furthermore, the notion of nano $h$-open functions, nano $h$-irresolute functions, nano $h$-totally continuous functions, nano $h$-contra continuous functions and nano $h$-homeomorphism have been put forth. Some properties regarding these functions have been investigated and some remarks related to them have been provided, supported by examples.

\end{abstract}

\footnote {2010 Mathematics Subject Classification: }

\maketitle{}
\section{Introduction}
hivagar and Richard\cite{Thiva 1} proposed the idea of nano topology, which they characterised in terms of approximations(lower and upper) and the boundary region of a subset of the universe using an equivalence relation. The elements of a nano topological space are called the nano-open sets \cite{Thiva 1}. Thivagar and Richard \cite{Thiva 2} explored several weak forms of nano-open sets, including nano $\alpha$-open sets, nano semi-open sets, and nano pre-open sets. Several mathematicians focused on the extensions of nano open sets after Thivagar and Richard's work on nano weakly open sets(see, for instance, \cite{Reva}). In 2021, Abbas \cite{Abbas} presented a new class of open sets in a topological space, called $h$-open sets. Utilizing the idea of $h$-open sets, we generalise nano open sets and introduce a new class of nano open sets in nano topological spaces, namely nano $h$-open sets. Next, some fundamental notions and properties regarding this class has also been put forth and studied in detail.       

\section{Preliminaries}
This section covers some fundamental definitions that will be utilized in the subsequent sections. 

\begin{definition} \cite{Abbas}
	Let $(\mathcal X, \tau)$ be a topological space and $\mathcal B $ a subset of $\mathcal X$. Then $ \mathcal B$ is said to be $h$-open if $\mathcal B \subseteq Int(\mathcal B \cup \mathcal O)$ for every non-empty open set $\mathcal O$ in $\mathcal X$, where $\mathcal O \neq \mathcal X$.
\end{definition}

\begin{definition} \cite{Thiva 1}
	Consider a non-empty finite set of objects $\mathcal U$ called the universe and an equivalence relation $\mathcal R$ on $\mathcal U$. The pair $(\mathcal U, \mathcal R)$ is called as an approximation space. Let $\mathcal X \subseteq \mathcal U$:
	
	1. The lower approximation of $\mathcal X$ with respect to $\mathcal R$ is denoted by $ \mathcal L_{\mathcal R}(\mathcal X)$ and $ \mathcal L_{\mathcal R}( \mathcal X)= \cup_{x \in \mathcal U} \{ \mathcal R(x): \mathcal R(x) \subseteq \mathcal X \}$, where $ \mathcal R(x)$ denotes the equivalence class determined by $x$.
	
	2. The upper approximation of $\mathcal X$ with respect to $\mathcal R$ is denoted by $\mathcal U_ \mathcal R(\mathcal X)$ and $\mathcal U_ \mathcal R(\mathcal X)=\cup_{x \in \mathcal U} \{ \mathcal R(x): \mathcal R(x) \cap \mathcal X \neq \phi \}$.
	
	3. The boundary region of $\mathcal X$ with respect to $\mathcal R$ is denoted by $\mathcal B_{\mathcal R}(\mathcal X)$ and $\mathcal B_{\mathcal R}(X)=\mathcal U_{\mathcal R}(\mathcal X) \setminus \mathcal L_{\mathcal R}(\mathcal X)$.
\end{definition}

According to Pawlak's definitions, $\mathcal X$ is called a rough set if $ \mathcal L_{\mathcal R}(\mathcal X) \neq  \mathcal U_{\mathcal R}(\mathcal X)$. 

\begin{definition} \cite{Thiva 1}
	Consider the universe $\mathcal U$ and an equivalence relation $\mathcal R$ on $\mathcal U$. Then for $\mathcal X \subseteq \mathcal U$, $\tau_{\mathcal R}(\mathcal X)= \{ \phi, \mathcal U, \mathcal L_{\mathcal R}(\mathcal X) , \mathcal U_{\mathcal R}(\mathcal X), \mathcal B_{\mathcal R}(\mathcal X) \}$ is called the nano topology on $\mathcal U$. $(\mathcal U, \tau_{\mathcal R}(\mathcal X))$ is called a nano topological space. The elements of $\tau_{\mathcal R}(\mathcal X) $ are called as nano open sets and the complement of a nano open set is called a nano closed set.
\end{definition}

\begin{definition} \cite{Thiva 1}
	Consider a nano topological space $(\mathcal U, \tau_{\mathcal R}(\mathcal X))$ with respect to $\mathcal X$, where $\mathcal X \subseteq \mathcal U$. Let $ \mathcal B \subseteq \mathcal U$, then\\
	(i) the nano interior of the set $\mathcal B$ is defined as the union of all nano open subsets contained in $ \mathcal B$, and is denoted by $nInt(\mathcal B)$;\\
	(ii) the nano closure of the set $\mathcal B$ is defined as the intersection of all nano closed subsets containing $\mathcal B$, and is denoted by $nCl(\mathcal B)$.
\end{definition}

\begin{note}
	From now onwards, $\mathcal{U}$ and $\mathcal{V}$ are two non-empty universes containing the sets $\mathcal{X}$ and $\mathcal{Y}$ respectively. Also, $\mathcal{U} \slash \mathcal{R}$ and $\mathcal{V} \slash \mathcal{R}^{'}$ denotes the families of equivalence classes by equivalence relations $\mathcal{R}$ and $\mathcal{R}^{'}$ in $\mathcal{U}$ and $\mathcal{V}$, respectively. Further, $(\mathcal{U}, \tau_{\mathcal{R}}(\mathcal{X}))$ and $(\mathcal{V}, \tau_{\mathcal{R}^{'}}(\mathcal{Y}))$ denote two nano topological spaces with respect to $\mathcal{X}$ and $\mathcal{Y}$, respectively. 
\end{note}

\begin{definition} 
	Consider two nano topological spaces $(\mathcal{U}, \tau_{\mathcal{R}}(\mathcal{X}))$ and $(\mathcal{V}, \tau_{\mathcal{R}^{'}}(\mathcal{Y}))$. Then a mapping $\psi: (\mathcal{U}, \tau_{\mathcal{R}}(\mathcal{X})) \rightarrow (\mathcal{V}, \tau_{\mathcal{R}^{'}}(\mathcal{Y}))$ is said to be \\
	(i) nano continuous \cite{Thiva 2} if inverse image of every nano open set in $\mathcal V$ is nano open in $\mathcal U$;\\
	(ii) nano open \cite{Thiva 2} if image of every nano open set in $\mathcal U$ is nano open in $\mathcal V$;\\
	(iii) nano homeomorphism \cite{Thiva 2} if $\psi$ is bijective, nano continuous and nano open;\\
	(iv) nano totally continuous \cite{Karth} if inverse image of every nano open set in $\mathcal{V}$ is nano clopen in $\mathcal U$;\\
	(v) nano contra continuous \cite{Thiva 3} if inverse image of every nano open set in $\mathcal{V}$ is nano closed in $\mathcal U$.
\end{definition}

\section{Nano h-open sets and nano h-continuity}

\begin{definition}
	Consider a nano topological space $(\mathcal{U}, \tau_{\mathcal{R}}(X))$ and a subset $\mathcal{B}$ of $\mathcal{U}$. Then $\mathcal{B}$ is said to be \textbf{nano $h$-open} if $\mathcal{B} \subseteq nInt(\mathcal{B}\cup \mathcal{O})$ for every non-empty nano open set $\mathcal{O}$ in $\mathcal{U}$ and $\mathcal{O} \neq \mathcal{U}$. Moreover, complement of a nano $h$-open set is called a nano $h$-closed set.
\end{definition}

\begin{note}
	$\tau_{\mathcal{R}}^{h}(\mathcal X)$ denotes the collection of all nano $h$-open sets in a nano topological space $(\mathcal{U}, \tau_{\mathcal{R}}( \mathcal X))$.
\end{note}

\begin{example}\label{Ex 1}
	Let $\mathcal{U}=\{a, b, c, d\}$, $\mathcal{U} \slash \mathcal{R}= \{ \{a\}, \{d\}, \{b, c\} \}$ and $\mathcal{X}= \{a, d\}$. Then $\tau_{\mathcal{R}}(\mathcal{X}) =\{U, \phi, \{a, d\}\}$. Clearly, $\{a\}$ is nano $h$-open set in $(\mathcal{U}, \tau_{\mathcal{R}}(\mathcal X))$.
\end{example}

\begin{theorem} \label{Thm 1}
	Every nano open set $\mathcal{B}$ in a nano topological space $(\mathcal{U}, \tau_{\mathcal{R}}(\mathcal X))$ is nano $h$-open in $(\mathcal{U}, \tau_{\mathcal{R}}( \mathcal X))$.
\end{theorem}

\begin{proof}
	Since $\mathcal{B}$ is nano open, $\mathcal{B} \subseteq nInt \mathcal{B} \subseteq nInt(\mathcal{B}\cup \mathcal{O})$ for every non-empty nano open set $\mathcal{O}$ in $\mathcal{U}$ and $\mathcal{O} \neq \mathcal{U}$. Thus $\mathcal{B}$ is nano $h$-open.
\end{proof}

\begin{remark}
	The converse of Theorem \ref{Thm 1} need not be true, as we can see in Example \ref{Ex 1}, $\{a\}$ is nano $h$-open set in $(\mathcal{U}, \tau_{\mathcal{R}}(\mathcal X))$ but not nano open.  
\end{remark}

\begin{theorem} \label{Thm 2}
	Consider a nano topological space $(\mathcal{U}, \tau_{\mathcal{R}}(\mathcal X))$. Let $\mathcal{B}_{1}$ and $\mathcal{B}_{2}$ be any two nano $h$-open sets in $(\mathcal{U}, \tau_{\mathcal{R}}( \mathcal X))$. Then: 
	\begin{enumerate}
		\item Intersection of $\mathcal{B}_{1}$ and $\mathcal{B}_{2}$ is again nano $h$-open;
		\item Union of $\mathcal{B}_{1}$ and $\mathcal{B}_{2}$ is again nano $h$-open.
	\end{enumerate}
\end{theorem}

\begin{proof}
	(1) Since $\mathcal{B}_{1}$ and $\mathcal{B}_{2}$ are nano $h$-open, $\mathcal{B}_{1} \subseteq nInt(\mathcal{B}_{1} \cup \mathcal{O})$ and $\mathcal{B}_{2} \subseteq nInt(\mathcal{B}_{2} \cup \mathcal{O})$ for every non-empty nano open set $\mathcal{O}$ in $\mathcal{U}$ and $\mathcal{O} \neq \mathcal{U}$. Then $\mathcal{B}_{1} \cap \mathcal{B}_{2} \subseteq nInt(\mathcal{B}_{1} \cup \mathcal{O}) \cap nInt(\mathcal{B}_{2} \cup \mathcal{O}) = nInt((\mathcal{B}_{1} \cup \mathcal{O}) \cap (\mathcal{B}_{2} \cup \mathcal{O})) = nInt(( \mathcal{B}_{1} \cap \mathcal{B}_{2}) \cup \mathcal{O})$. Thus $\mathcal{B}_{1} \cap \mathcal{B}_{2}$ is nano $h$-open. \\
	(2) Since $\mathcal{B}_{1}$ and $\mathcal{B}_{2}$ are nano $h$-open, $\mathcal{B}_{1} \subseteq nInt(\mathcal{B}_{1} \cup \mathcal{O})$ and $\mathcal{B}_{2} \subseteq nInt(\mathcal{B}_{2} \cup \mathcal{O})$ for every non-empty nano open set $\mathcal{O}$ in $\mathcal{U}$ and $\mathcal{O} \neq \mathcal{U}$. Then $\mathcal{B}_{1} \cup \mathcal{B}_{2} \subseteq nInt(\mathcal{B}_{1} \cup \mathcal{O}) \cup nInt(\mathcal{B}_{2} \cup \mathcal{O}) \subseteq nInt((\mathcal{B}_{1} \cup \mathcal{O}) \cup (\mathcal{B}_{2} \cup \mathcal{O})) = nInt(( \mathcal{B}_{1} \cup \mathcal{B}_{2}) \cup \mathcal{O})$. Thus $\mathcal{B}_{1} \cup \mathcal{B}_{2}$ is nano $h$-open.
\end{proof}

\begin{corollary}
	Consider a nano topological space $(\mathcal{U}, \tau_{\mathcal{R}}(\mathcal X))$. Let $\mathcal{B}_{1}$ be a nano open set and $\mathcal{B}_{2}$ be a nano $h$-open set in $(\mathcal{U}, \tau_{\mathcal{R}}(\mathcal X))$. Then: 
	\begin{enumerate}
		\item Intersection of $\mathcal{B}_{1}$ and $\mathcal{B}_{2}$ is again nano $h$-open;
		\item Union of $\mathcal{B}_{1}$ and $\mathcal{B}_{2}$ is again nano $h$-open.
	\end{enumerate}
\end{corollary}

\begin{proof}
	Obvious from Theorem \ref{Thm 1}.
\end{proof}

\begin{definition}
	Consider a nano topological space $(\mathcal{U}, \tau_{\mathcal{R}}(\mathcal X))$ and a subset $\mathcal{B}$ of $\mathcal{U}$. Then \textbf{nano $h$-interior} of $\mathcal{B}$ is denoted by \textbf{$nInt_{h}(\mathcal{B})$} and defined as $$ nInt_{h}(\mathcal{B})= \bigcup \{ \mathcal{C} : \mathcal{C} ~is~ nano~ h-open~ in~ (\mathcal{U}, \tau_{\mathcal{R}}(\mathcal X))~ and ~ \mathcal{C} \subseteq \mathcal{B} \}$$ i.e. $nInt_{h}(\mathcal{B})$ is the largest nano $h$-open set in $(\mathcal{U}, \tau_{\mathcal{R}}(\mathcal X))$ contained in $\mathcal{B}$. Clearly, $nInt_{h}(\mathcal{B})$ is nano $h$-open set.
\end{definition}

\begin{definition}
	Consider a nano topological space $(\mathcal{U}, \tau_{\mathcal{R}}(\mathcal X))$ and a subset $\mathcal{B}$ of $\mathcal{U}$. Then \textbf{nano $h$-closure} of $\mathcal{B}$ is denoted by \textbf{$nCl_{h}(\mathcal{B})$} and defined as $$ nCl_{h}(\mathcal{B})= \bigcap \{ \mathcal{C} : \mathcal{C} ~is~ nano~ h-closed~ in~ (\mathcal{U}, \tau_{\mathcal{R}}(X))~and~\mathcal{B} \subseteq \mathcal{C}\}$$ i.e. $nCl_{h}(\mathcal{B})$ is the smallest nano $h$-closed set in $(\mathcal{U}, \tau_{\mathcal{R}}(X))$ containing $\mathcal{B}$. Clearly, $nCl_{h}(\mathcal{B})$ is nano $h$-closed set.
\end{definition}

\begin{proposition}
	Consider a nano topological space $(\mathcal{U}, \tau_{\mathcal{R}}(\mathcal X))$. Let $\mathcal{B}_{1}$ and $\mathcal{B}_{2}$ be any two subsets of $(\mathcal{U}, \tau_{\mathcal{R}}(\mathcal X))$ such that $\mathcal{B}_{1} \subseteq \mathcal{B}_{2}$. Then:
	\begin{enumerate}
		\item $nInt_{h}(\mathcal{B}_{1}) \subseteq nInt_{h}(\mathcal{B}_{2})$;
		\item $nCl_{h}(\mathcal{B}_{1}) \subseteq nCl_{h}(\mathcal{B}_{2})$.
	\end{enumerate}
\end{proposition}

\begin{proposition}
	Consider a nano topological space $(\mathcal{U}, \tau_{\mathcal{R}}( \mathcal X))$ and a subset $\mathcal{B}$ of $\mathcal{U}$. Then: 
	\begin{enumerate}
		\item $nInt_{h}(\mathcal{B}) \subseteq \mathcal{B}$; 
		\item $\mathcal{B} \subseteq nCl_{h}(\mathcal{B})$;
		\item $\mathcal{B}$ is nano $h$-open $\Leftrightarrow \mathcal{B} = nInt_{h}(\mathcal{B})$;
		\item $\mathcal{B}$ is nano $h$-closed $\Leftrightarrow \mathcal{B} = nCl_{h}(\mathcal{B})$.
	\end{enumerate}
\end{proposition}

\begin{definition}
	Consider two nano topological spaces $(\mathcal{U}, \tau_{\mathcal{R}}(\mathcal{X}))$ and $(\mathcal{V}, \tau_{\mathcal{R}^{'}}(\mathcal{Y}))$. Then a mapping $\psi: (\mathcal{U}, \tau_{\mathcal{R}}(\mathcal{X})) \rightarrow (\mathcal{V}, \tau_{\mathcal{R}^{'}}(\mathcal{Y}))$ is said to be \textbf{nano $h$-continuous} if $\psi^{-1}(\mathcal{O})$ is nano $h$-open in $\mathcal{U}$ for every nano open set $\mathcal{O}$ in $\mathcal{V}$.
\end{definition}

\begin{example} \label{Ex 2}
	Let $\mathcal{U}=\mathcal{V}=\{a,b,c\}$ with $\mathcal{U} \slash \mathcal{R}= \{ \{a\}, \{b\}, \{c\} \}$ and $\mathcal{V} \slash \mathcal{R}= \{ \{a,b\}, \{c\} \}$. Let $\mathcal{X}= \{a,c \} \subseteq \mathcal{U}$  and $\mathcal{Y}=\{a,c\} \subseteq \mathcal{V}$. Then $\tau_{\mathcal{R}} (\mathcal{X})= \{ \phi, \mathcal{U}, \{a,c\} \}$ and $\tau_{\mathcal{R}^{'}} (\mathcal{Y})= \{ \phi, \mathcal{V}, \{c\}, \{a,b\}\}$. Further, $\tau_{\mathcal{R}}^{h}(\mathcal{X})= \mathcal{P}(\mathcal{U})$. Define $\psi: (\mathcal{U}, \tau_{\mathcal{R}}(\mathcal{X})) \rightarrow (\mathcal{V},  \tau_{\mathcal{R}^{'}}(\mathcal{Y}))$ to be an identity mapping. Then $\psi$ is nano $h$-continuous.
\end{example}

\begin{theorem} \label{Thm 3}
	Suppose a mapping $\psi: (\mathcal{U}, \tau_{\mathcal{R}}(\mathcal{X})) \rightarrow (\mathcal{V}, \tau_{\mathcal{R}^{'}}(\mathcal{Y}))$ is nano continuous. Then $\psi$ is nano $h$-continuous.
\end{theorem}
\begin{proof}
	Suppose $\psi: (\mathcal{U}, \tau_{\mathcal{R}}(\mathcal{X})) \rightarrow (\mathcal{V}, \tau_{\mathcal{R}^{'}}(\mathcal{Y}))$ be a nano continuous function and $\mathcal{O}$ be any nano open set in $\mathcal{V}$. By hypothesis, $\psi^{-1}(\mathcal{O})$ is nano open in $\mathcal{U}$ and by Theorem \ref{Thm 1}, $\psi^{-1}(\mathcal{O})$ is nano $h$-open in $\mathcal{U}$. Thus $\psi: (\mathcal{U}, \tau_{\mathcal{R}}(\mathcal{X})) \rightarrow (\mathcal{V}, \tau_{\mathcal{R}^{'}}(\mathcal{Y}))$ is nano $h$-continuous. 
\end{proof}

\begin{remark}
	Converse of Theorem \ref{Thm 3} need not be true, as we can see in Example \ref{Ex 2}, $\{c\}$ is nano open in $\mathcal V$ but $\psi^{-1}(\{c\})=\{ c \}$ is not nano open in $\mathcal U$. 
\end{remark}

\begin{theorem} \label{Thm 4}
	Consider a function $\psi: (\mathcal{U}, \tau_{\mathcal{R}}(\mathcal{X})) \rightarrow (\mathcal{V}, \tau_{\mathcal{R}^{'}}(\mathcal{Y}))$. Then the following are equivalent: 
	\begin{itemize}
		\item[(1)] $\psi$ is nano $h$-continuous.
		\item[(2)] $\psi^{-1}(\mathcal K)$ is nano $h$-closed set in $\mathcal{U}$, for every nano closed set $\mathcal K$ in $\mathcal V$.
		\item[(3)] $\psi(nCl_{h}(\mathcal B)) \subseteq nCl(\psi(\mathcal{B}))$, for each subset $\mathcal{B}$ of $\mathcal{U}$.
		\item[(4)] $nCl_{h}(\psi^{-1}(\mathcal C)) \subseteq \psi^{-1}(nCl(\mathcal C))$, for each subset $\mathcal C$ of $\mathcal V$.
		\item[(5)] $\psi^{-1}(nInt (\mathcal C)) \subseteq nInt_{h}(\psi^{-1}(\mathcal C))$, for each subset $\mathcal C$ of $\mathcal V$.
	\end{itemize}
\end{theorem}

\begin{proof}
	$(1) \Rightarrow (2)$ Suppose $\psi$ is nano $h$-continuous and $\mathcal{K}$ is nano closed in $\mathcal V$. Then $\mathcal V \setminus \mathcal K$ is nano open in $V$. As $\psi$ is nano $h$-continuous, $\psi^{-1}( \mathcal V \setminus  \mathcal K)$ is nano $h$-open in $\mathcal U$. Thus $\mathcal V \setminus \psi^{-1}( \mathcal K)$ is nano $h$-open in $\mathcal U$ and hence $\psi^{-1}( \mathcal K)$ is nano $h$-closed in $\mathcal U$. \\
	
	$(2) \Rightarrow (1)$ Let $\mathcal O$ be a nano open set in $\mathcal V$. Then $\mathcal V \setminus \mathcal O$ is nano closed in $\mathcal V$ and $\psi^{-1}( \mathcal V \setminus  \mathcal O)$ is nano $h$-closed in $\mathcal U$. Thus $\psi^{-1}( \mathcal O)$ is nano $h$-open in $\mathcal U$. Hence $\psi$ is nano $h$-continuous. \\
	
	$(1) \Rightarrow (3)$ Suppose $\psi$ is nano $h$-continuous. Clearly, $nCl(\psi(\mathcal B))$ is nano closed in $\mathcal V$ and thus by (2), $\psi^{-1}(nCl(\psi(\mathcal B)))$ is nano $h$-closed in $\mathcal U$. Also, $\psi(\mathcal B) \subseteq nCl(\psi(\mathcal B)) $ implies $\psi^{-1}(\psi(\mathcal B)) \subseteq \psi^{-1}(nCl(\psi(\mathcal B)))$. This further implies that $nCl_{h}(\mathcal B) \subseteq nCl_{h}[\psi^{-1}(nCl(\psi(\mathcal B)))] = \psi^{-1}(nCl(\psi(\mathcal B)))$. Consequently, $nCl_{h}(\mathcal B) \subseteq \psi^{-1}(nCl(\psi(\mathcal B)))$ and hence $\psi(nCl_{h}(\mathcal B) \subseteq nCl(\psi(\mathcal{B}))$. \\
	
	$(3) \Rightarrow (1)$ Suppose  $\psi(nCl_{h}(\mathcal B)) \subseteq nCl(\psi(\mathcal{B}))$, for each subset $\mathcal{B}$ of $\mathcal{U}$. Consider a nano closed set $\mathcal K$ in $\mathcal V$. Now $\psi(nCl_{h}(\psi^{-1}(\mathcal K))) \subseteq nCl(\psi( \psi^{-1}(\mathcal K)))=nCl( \mathcal K)= \mathcal K$. Thus, $nCl_{h}(\psi^{-1}(\mathcal K)) \subseteq \psi^{-1}(\mathcal K)$. Also $\psi^{-1}(\mathcal K) \subseteq nCl_{h}(\psi^{-1}(\mathcal K))$. Consequently, $nCl_{h}(\psi^{-1}(\mathcal K)) = \psi^{-1}(\mathcal K)$. This implies that $\psi^{-1}(\mathcal K)$ is nano $h$-closed in $\mathcal U$ and hence $\psi$ is nano $h$-continuous.\\
	
	$(1) \Rightarrow (4)$ Suppose $\psi$ is nano $h$-continuous. Since $nCl(\mathcal C)$ is nano closed in $\mathcal V$, $\psi^{-1}(nCl(\mathcal C))$ is nano $h$-closed in $\mathcal U$. Thus $nCl_{h}[\psi^{-1}(nCl(\mathcal C))] =\psi^{-1}(nCl(\mathcal C))$. Further $\mathcal C \subseteq nCl(\mathcal C)$ implies $\psi^{-1}(\mathcal C) \subseteq \psi^{-1}(nCl(\mathcal C))$. Consequently, $nCl_{h}(\psi^{-1}(\mathcal C)) \subseteq \psi^{-1}(nCl(\mathcal C))$.\\
	
	$(4) \Rightarrow (1)$ Let us suppose $nCl_{h}(\psi^{-1}(\mathcal C)) \subseteq \psi^{-1}(nCl(\mathcal C))$, for each $\mathcal C \subseteq \mathcal V$. To prove $\psi$ is nano $h$-continuous, consider a nano closed set $\mathcal K$ in $\mathcal V$. Then $nCl(\mathcal K)= \mathcal K$. By hypothesis, $nCl_{h}(\psi^{-1}(\mathcal K)) \subseteq \psi^{-1}(nCl(\mathcal K))= \psi^{-1}(\mathcal K)$. Also, $\psi^{-1}(\mathcal K) \subseteq nCl_{h}(\psi^{-1}(\mathcal K))$. Consequently, $\psi^{-1}(\mathcal K) =nCl_{h}(\psi^{-1}(\mathcal K))$. This implies that $\psi^{-1}(\mathcal K)$ is nano $h$-closed set in $\mathcal U$. This completes the proof. \\
	
	$(1) \Rightarrow (5)$ Suppose $\psi$ is nano $h$-continuous. Since $nInt(\mathcal C)$ is nano open in $\mathcal V$, $\psi^{-1}(nInt(\mathcal C))$ is nano $h$-open in $\mathcal U$. Consequently, $nInt_{h}(\psi^{-1}(nInt \mathcal C)) =\psi^{-1}(nInt \mathcal C)$. Also, $\psi^{-1}(nInt(\mathcal C)) \subseteq \psi^{-1}(\mathcal C)$ follows from $nInt \mathcal C \subseteq \mathcal C$. This implies that $nInt_{h}(\psi^{-1}(nInt (\mathcal C))) \subseteq nInt_{h}(\psi^{-1}(\mathcal C))$. This further implies that $\psi^{-1}(nInt \mathcal C) \subseteq nInt_{h}(\psi^{-1}(\mathcal C))$. \\
	
	$(5) \Rightarrow (1)$ Suppose $\psi^{-1}(nInt \mathcal C) \subseteq nInt_{h}(\psi^{-1}(\mathcal C))$ for each subset $\mathcal C \subseteq \mathcal V$. To prove $\psi$ is nano $h$-continuous, consider a nano open set $\mathcal O$ in $\mathcal V$. Then $nInt(\mathcal O)= \mathcal O$. By hypothesis, $\psi^{-1}(nInt(\mathcal O)) \subseteq nInt_{h}(\psi^{-1}(\mathcal O))$. Thus $\psi^{-1}(\mathcal O) \subseteq nInt_{h}(\psi^{-1}(\mathcal O))$. Also, $nInt_{h}(\psi^{-1}(\mathcal O)) \subseteq \psi^{-1}(\mathcal O)$. Therefore, $\psi^{-1}(\mathcal O)= nInt_{h}(\psi^{-1}(\mathcal O))$. Consequently, $\psi^{-1}(\mathcal O)$ is nano $h$-open set in $\mathcal U$ and thus $\psi$ is nano $h$-continuous. 
\end{proof}

\begin{note}
	Equality need not hold in $(3), (4)$ and $(5)$ of Theorem \ref{Thm 4}. It can be illustrated by the following example:
\end{note}

\begin{example} \label{Ex 3}
	Let $\mathcal U=\{a,b,c \}$ with $\mathcal U \slash \mathcal R= \{ \{a\}, \{b\}, \{c\}\}$, $\mathcal X= \{a\}$. Then $\tau_{\mathcal R}(\mathcal X)= \{ \phi, \mathcal U, \{a\}\}$, $\tau_{\mathcal R}^{h}(\mathcal X) = \{ \phi, \mathcal U, \{a\}, \{b,c\}\}$. Let $\mathcal V= \{1,2,3\}$, $\mathcal V \slash \mathcal R^{'}=\{ \{1\}, \{2,3\}\}$, $\mathcal Y= \{2,3\}$. Then $\tau_{\mathcal R^{'}}(\mathcal Y)= \{ \phi, \mathcal V, \{2,3\} \}$, $\tau^{h}_{\mathcal R^{'}}(\mathcal Y) = \mathcal P(\mathcal V)$. Define  $\psi: (\mathcal{U}, \tau_{\mathcal{R}}(\mathcal{X})) \rightarrow (\mathcal{V}, \tau_{\mathcal{R}^{'}}(\mathcal{Y}))$ as $\psi(\{a\})=\{1\}, \psi(\{b\})=\{2\}$ and $\psi(\{c\})=\{3\}$. Clearly, $\psi$ is nano $h$-continuous. 
	\begin{enumerate}
		\item Let $\mathcal B= \{b,c\} \subseteq \mathcal U$. Now $\psi(nCl_{h}(\{b,c\})= \psi(\{b,c\})=\{2,3\}$ but $nCl(\psi(\{b,c\}))= nCl(\{2,3\}) =\mathcal V$. Thus $\psi(nCl_{h}(\mathcal B) \neq nCl(\psi(\mathcal{B}))$.
		\item Let $\mathcal C = \{2,3\} \subseteq \mathcal V$. Now $nCl_{h}(\psi^{-1}(\{2,3\}))= nCl_{h}(\{b,c\})=\{b,c\}$ but $\psi^{-1}(nCl(\{2,3\}))= \psi^{-1}(\mathcal V)= \mathcal U$. Thus $nCl_{h}(\psi^{-1}(\mathcal C)) \neq \psi^{-1}(nCl(\mathcal C))$. 
		\item Let $\mathcal C= \{1\} \subseteq \mathcal V$. Now $\psi^{-1}(nInt\{1\})= \psi^{-1}(\phi)$ but $nInt_{h}(\psi^{-1}(\{1\}))= nInt_{h}(\{a\})=\{a\}$. Thus $\psi^{-1}(nInt (\mathcal C)) \neq nInt_{h}(\psi^{-1}(\mathcal C))$.
	\end{enumerate}
\end{example}

\begin{definition}
	Consider two nano topological spaces $(\mathcal{U}, \tau_{\mathcal{R}}(\mathcal{X}))$ and $(\mathcal{V}, \tau_{\mathcal{R}^{'}}(\mathcal{Y}))$. Then a mapping $\psi: (\mathcal{U}, \tau_{\mathcal{R}}(\mathcal{X})) \rightarrow (\mathcal{V}, \tau_{\mathcal{R}^{'}}(\mathcal{Y}))$ is said to be \textbf{nano $h$-open} if $\psi(\mathcal{O})$ is nano $h$-open in $\mathcal{V}$ for every nano open set $\mathcal{O}$ in $\mathcal{U}$.
\end{definition}

\begin{example}\label{Ex 4}
	Let $\mathcal U= \mathcal V= \{a,b,c\}$ with $\mathcal U \slash \mathcal R= \{ \{a, b\}, \{c\}\}$, $\mathcal V \slash \mathcal R^{'}=\{ \{a\}, \{b\},\{c\}\}$. Let $\mathcal X =\{a,c\}$, $\mathcal Y= \{a,c\}$. Then $\tau_{\mathcal R}(\mathcal X)= \{ \phi, \mathcal U, \{c\},\{a,b\} \}$ and $\tau_{\mathcal R^{'}}(\mathcal Y)= \{ \phi, \mathcal V, \{a,c\} \}$. Thus $\tau^{h}_{\mathcal R^{'}}(\mathcal Y) = \mathcal P(\mathcal V)$. Clearly, the identity function $\psi:(\mathcal{U}, \tau_{\mathcal{R}}(\mathcal{X})) \rightarrow (\mathcal{V}, \tau_{\mathcal{R}^{'}}(\mathcal{Y}))$ is nano $h$-open.
\end{example}

\begin{theorem} \label{Thm 5}
	Suppose a mapping $\psi: (\mathcal{U}, \tau_{\mathcal{R}}(\mathcal{X})) \rightarrow (\mathcal{V}, \tau_{\mathcal{R}^{'}}(\mathcal{Y}))$ is nano open. Then $\psi$ is nano $h$-open.
\end{theorem}

\begin{proof}
	To prove $\psi$ is nano $h$-open, consider a nano open set $\mathcal O$ in $\mathcal U$. By hypothesis, $\psi(\mathcal O)$ is nano open in $\mathcal V$ and by Theorem \ref{Thm 1}, $\psi(\mathcal O)$ is nano $h$-open in $\mathcal V$. This completes the proof.
\end{proof}

\begin{remark}
	Converse of Theorem \ref{Thm 5} need not be true, as we can see in Example \ref{Ex 4}, $\{c\}$ is nano open in $\mathcal U$ but $\psi(\{c\})=\{c\}$ is not nano open in $\mathcal V$.
\end{remark}

\begin{definition}
	Consider two nano topological spaces $(\mathcal{U}, \tau_{\mathcal{R}}(\mathcal{X}))$ and $(\mathcal{V}, \tau_{\mathcal{R}^{'}}(\mathcal{Y}))$. Then a mapping $\psi: (\mathcal{U}, \tau_{\mathcal{R}}(\mathcal{X})) \rightarrow (\mathcal{V}, \tau_{\mathcal{R}^{'}}(\mathcal{Y}))$ is said to be nano h-irresolute if $\psi^{-1}(\mathcal{O})$ is nano $h$-open in $\mathcal{U}$ for every nano $h$-open set $\mathcal{O}$ in $\mathcal{V}$.
\end{definition}

\begin{example} \label{Ex 5}
	Let $\mathcal U= \mathcal V = \{a,b,c\}$ with $\mathcal U \slash \mathcal R= \mathcal V \slash \mathcal R^{'} = \{ \{a,b\}, \{c\}\}$. Let $\mathcal X= \mathcal Y = \{a,b\}$. Then $\tau_{\mathcal R}(\mathcal X)= \{ \phi, \mathcal U, \{a,b\} \}$, $\tau_{\mathcal R^{'}}(\mathcal Y)= \{ \phi, \mathcal V, \{a,b\} \}$. Also, $\tau_{\mathcal R}^{h}(\mathcal X) = \mathcal P(\mathcal U)$ and $\tau_{\mathcal R^{'}}^{h}(\mathcal Y) = \mathcal P(\mathcal V)$. Clearly, the identity function $\psi:(\mathcal{U}, \tau_{\mathcal{R}}(\mathcal{X})) \rightarrow (\mathcal{V}, \tau_{\mathcal{R}^{'}}(\mathcal{Y}))$ is nano $h$-irresolute function.
\end{example}

\begin{theorem} \label{Thm 6}
	Suppose a mapping $\psi: (\mathcal{U}, \tau_{\mathcal{R}}(\mathcal{X})) \rightarrow (\mathcal{V}, \tau_{\mathcal{R}^{'}}(\mathcal{Y}))$ is nano h-irresolute. Then $\psi$ is nano $h$-continuous.
\end{theorem}

\begin{proof}
	Consider a nano open set $\mathcal O$ in $\mathcal V$. By Theorem \ref{Thm 1}, $\mathcal O$ is nano $h$-open set in $\mathcal V$. By hypothesis, $\psi^{-1}(\mathcal O)$ is nano $h$-open set in $\mathcal U$. This completes the proof.
\end{proof}

\begin{remark}
	Converse of Theorem \ref{Thm 6} need not be true, as we can see in Example \ref{Ex 3}, $\{2\}$ is nano $h$-open in $\mathcal V$ but $\psi^{-1}(\{2\})=\{b\}$ is not nano $h$-open in $\mathcal U$. 
\end{remark}

\begin{definition}
	Consider two nano topological spaces $(\mathcal{U}, \tau_{\mathcal{R}}(\mathcal{X}))$ and $(\mathcal{V}, \tau_{\mathcal{R}^{'}}(\mathcal{Y}))$. Then a mapping $\psi: (\mathcal{U}, \tau_{\mathcal{R}}(\mathcal{X})) \rightarrow (\mathcal{V}, \tau_{\mathcal{R}^{'}}(\mathcal{Y}))$ is said to be nano $h$-homeomorphism if $\psi$ is bijective, nano $h$-continuous and nano $h$-open function.
\end{definition} 

\begin{example}
	In Example \ref{Ex 3}, we can see the identity function $\psi: (\mathcal{U}, \tau_{\mathcal{R}}(\mathcal{X})) \rightarrow (\mathcal{V}, \tau_{\mathcal{R}^{'}}(\mathcal{Y}))$ is bijective, nano $h$-continuous and nano $h$-open function and hence, nano $h$-homeomorphism. 
\end{example}

\begin{theorem} \label{Thm 7}
	Suppose a mapping $\psi: (\mathcal{U}, \tau_{\mathcal{R}}(\mathcal{X})) \rightarrow (\mathcal{V}, \tau_{\mathcal{R}^{'}}(\mathcal{Y}))$ is nano homeomorphism. Then $\psi$ is nano $h$-homeomorphism.
\end{theorem}

\begin{proof}
	Proof follows trivially from Theorem \ref{Thm 3} and \ref{Thm 5}.
\end{proof}

\begin{remark}
	Converse of Theorem \ref{Thm 7} need not be true, as we can see in Example \ref{Ex 3}, $\{a\}$ is nano open in $\mathcal U$ but $\psi(\{a\})=\{1\}$ is not nano open in $\mathcal V$. Thus $\psi: (\mathcal{U}, \tau_{\mathcal{R}}(\mathcal{X})) \rightarrow (\mathcal{V}, \tau_{\mathcal{R}^{'}}(\mathcal{Y}))$  is nano $h$-open but not nano open. Therefore, we conclude that $\psi: (\mathcal{U}, \tau_{\mathcal{R}}(\mathcal{X})) \rightarrow (\mathcal{V}, \tau_{\mathcal{R}^{'}}(\mathcal{Y}))$ is nano $h$-homeomorphism but not nano homeomorphism.
\end{remark}

\begin{definition}
	Consider two nano topological spaces $(\mathcal{U}, \tau_{\mathcal{R}}(\mathcal{X}))$ and $(\mathcal{V}, \tau_{\mathcal{R}^{'}}(\mathcal{Y}))$. Then a mapping $\psi: (\mathcal{U}, \tau_{\mathcal{R}}(\mathcal{X})) \rightarrow (\mathcal{V}, \tau_{\mathcal{R}^{'}}(\mathcal{Y}))$ is said to be nano $h$-totally continuous if $\psi^{-1}(\mathcal{O})$ is nano clopen in $\mathcal{U}$ for every nano $h$-open set $\mathcal{O}$ in $\mathcal{V}$.
\end{definition}

\begin{example}
	Let $\mathcal U= \mathcal V =\{a,b,c\}$ with $\mathcal U \slash \mathcal R = \{ \{a\}, \{b,c \} \}$, $\mathcal X =\{a,c\}$. Then $\tau_{\mathcal R}(\mathcal X)= \{ \phi, \mathcal U, \{a\}, \{b,c\} \}$. Let $ \mathcal V \slash \mathcal R^{'}= \{ \{a\}, \{b\}, \{c\} \}$, $\mathcal Y= \{a\}$. Then $\tau_{\mathcal R^{'}}(\mathcal Y)=\{ \phi, \mathcal V, \{a\} \}$, $\tau_{\mathcal R^{'}}^{h}(\mathcal Y)=\{ \phi, \mathcal V, \{a\}, \{b,c\} \}$. Define $\psi: (\mathcal{U}, \tau_{\mathcal{R}}(\mathcal{X})) \rightarrow (\mathcal{V}, \tau_{\mathcal{R}^{'}}(\mathcal{Y}))$ as an identity function. Then $\psi$ is nano $h$-totally continuous. 
\end{example}

\begin{theorem} \label{Thm 8}
	Suppose a mapping $\psi: (\mathcal{U}, \tau_{\mathcal{R}}(\mathcal{X})) \rightarrow (\mathcal{V}, \tau_{\mathcal{R}^{'}}(\mathcal{Y}))$ is nano $h$-totally continuous. Then $\psi$ is nano totally continuous.
\end{theorem}

\begin{proof}
	To prove  $\psi$ is nano totally continuous, let $\mathcal O$ be any nano open set in $\mathcal V$. By Theorem \ref{Thm 1}, $\mathcal O$ is nano $h$-open in $\mathcal V$. Also, by hypothesis, $\psi^{-1}(\mathcal{O})$ is nano clopen in $\mathcal{U}$. Hence, the proof.
\end{proof}

\begin{remark}
	Converse of Theorem \ref{Thm 8} need not be true. It can be illustrated by the following example:
\end{remark}

\begin{example}
	Let $\mathcal U= \mathcal V =\{a,b,c\}$ with $\mathcal U \slash \mathcal R = \{ \{a\}, \{b,c \} \}$, $\mathcal X =\{a,c\}$. Then $\tau_{\mathcal R}(\mathcal X)= \{ \phi, \mathcal U, \{a\}, \{b,c\} \}$. Let $ \mathcal V \slash \mathcal R^{'}= \{ \{a\}, \{b,c \} \}$, $\mathcal Y= \{b,c\}$. Then $\tau_{\mathcal R^{'}}(\mathcal Y)=\{ \phi, \mathcal V, \{b,c\} \}$. Define $\psi: (\mathcal{U}, \tau_{\mathcal{R}}(\mathcal{X})) \rightarrow (\mathcal{V}, \tau_{\mathcal{R}^{'}}(\mathcal{Y}))$ as an identity function. Then $\psi$ is nano totally continuous but not nano $h$-totally continuous.
\end{example}

\begin{theorem} \label{Thm 9}
	Suppose a mapping $\psi: (\mathcal{U}, \tau_{\mathcal{R}}(\mathcal{X})) \rightarrow (\mathcal{V}, \tau_{\mathcal{R}^{'}}(\mathcal{Y}))$ is nano $h$-totally continuous. Then $\psi$ is nano $h$-irresolute.
\end{theorem}

\begin{proof}
	To prove  $\psi$ is nano $h$-irresolute, let $\mathcal O$ be any nano $h$-open set in $\mathcal V$. By hypothesis, $\psi^{-1}(\mathcal{O})$ is nano clopen in $\mathcal{U}$. By Theorem \ref{Thm 1}, $\psi^{-1}(\mathcal O)$ is nano $h$-open in $\mathcal U$. Hence, the proof.
\end{proof}

\begin{remark}
	Converse of Theorem \ref{Thm 9} need not be true, as we can see in Example \ref{Ex 5}, $\{a,c\}$ is nano $h$-open in $\mathcal V$ but $\psi^{-1}(\{a,c\})=\{a,c\}$ is not nano clopen in $\mathcal U$.
\end{remark}

\begin{definition}
	Consider two nano topological spaces $(\mathcal{U}, \tau_{\mathcal{R}}(\mathcal{X}))$ and $(\mathcal{V}, \tau_{\mathcal{R}^{'}}(\mathcal{Y}))$. Then a mapping $\psi: (\mathcal{U}, \tau_{\mathcal{R}}(\mathcal{X})) \rightarrow (\mathcal{V}, \tau_{\mathcal{R}^{'}}(\mathcal{Y}))$ is said to be nano $h$-contra continuous if $\psi^{-1}(\mathcal{O})$ is nano $h$-closed in $\mathcal{U}$ for every nano open set $\mathcal{O}$ in $\mathcal{V}$.
\end{definition}

\begin{example} \label{Ex 6}
	Let $\mathcal U=\{a,b,c \}$ with $\mathcal U \slash \mathcal R= \{ \{a\}, \{b\}, \{c\}\}$, $\mathcal X= \{b,c\}$. Then $\tau_{\mathcal R}(\mathcal X)= \{ \phi, \mathcal U, \{b,c\}\}$, $\tau_{\mathcal R}^{h}(\mathcal X) = \mathcal P(\mathcal U)$. Let $\mathcal V= \{1,2,3\}$, $\mathcal V \slash \mathcal R^{'}=\{ \{1\}, \{2\}, \{3\} \}$, $\mathcal Y= \{2\}$. Then $\tau_{\mathcal R^{'}}(\mathcal Y)= \{ \phi, \mathcal V, \{2\} \}$. Define  $\psi: (\mathcal{U}, \tau_{\mathcal{R}}(\mathcal{X})) \rightarrow (\mathcal{V}, \tau_{\mathcal{R}^{'}}(\mathcal{Y}))$ as $\psi(\{a\})=\{1\}, \psi(\{b\})=\{2\}$ and $\psi(\{c\})=\{3\}$. Clearly, $\psi$ is nano $h$- contra continuous.
\end{example}

\begin{theorem} \label{Thm 10}
	Suppose a mapping $\psi: (\mathcal{U}, \tau_{\mathcal{R}}(\mathcal{X})) \rightarrow (\mathcal{V}, \tau_{\mathcal{R}^{'}}(\mathcal{Y}))$ is nano contra continuous. Then $\psi$ is nano $h$-contra continuous.
\end{theorem}

\begin{proof}
	To prove  $\psi$ is nano $h$-contra continuous, let $\mathcal O$ be any nano open set in $\mathcal V$. By hypothesis, $\psi^{-1}(\mathcal{O})$ is nano closed in $\mathcal{U}$. Thus $\psi^{-1}(\mathcal{O})$ is nano $h$-closed in $\mathcal{U}$. Hence, the proof.
\end{proof}

\begin{remark}
	Converse of Theorem \ref{Thm 10} need not be true, as we can see in Example \ref{Ex 6}, $\{2\}$ is nano open in $\mathcal V$ but $\psi^{-1}(\{2\})=\{b\}$ is not nano closed in $\mathcal U$.
\end{remark}

\begin{theorem} \label{Thm 11}
	Suppose a mapping $\psi: (\mathcal{U}, \tau_{\mathcal{R}}(\mathcal{X})) \rightarrow (\mathcal{V}, \tau_{\mathcal{R}^{'}}(\mathcal{Y}))$ is nano totally continuous. Then $\psi$ is nano $h$-contra continuous.
\end{theorem}

\begin{proof}
	To prove  $\psi$ is nano $h$-contra continuous, let $\mathcal O$ be any nano open set in $\mathcal V$. By hypothesis, $\psi^{-1}(\mathcal{O})$ is nano clopen in $\mathcal{U}$. Thus $\psi^{-1}(\mathcal{O})$ is nano $h$-closed in $\mathcal{U}$. Hence, the proof.
\end{proof}

\begin{remark}
	Converse of Theorem \ref{Thm 11} need not be true, as we can see in Example \ref{Ex 6}, $\{2\}$ is nano open in $\mathcal V$ but $\psi^{-1}(\{2\})=\{b\}$ is not nano clopen in $\mathcal U$.
\end{remark}

\end{document}